\documentclass[11pt]{amsart}

\usepackage{color}

\usepackage{eucal}
\usepackage{amssymb,url,amscd,amsfonts}
\usepackage{latexsym}
\usepackage{hyperref}
\usepackage[all]{xy}
\usepackage{tikz}
\theoremstyle{plain}
\newtheorem{theorem}{Theorem}

\newtheorem{lemma}[theorem]{Lemma}

\newtheorem{proposition}[theorem]{Proposition}
\newtheorem{conjecture}[theorem]{Conjecture}

\numberwithin{equation}{section} 
\numberwithin{theorem}{section}

\begin{document}

\title
[An extension of Liebmann's Theorem to hypersurfaces with boundary]
{An extension of Liebmann's Theorem to hypersurfaces with boundary}

\author[Cruz, F. and Nelli, B. ]{Cruz, F. and Nelli, B.}
\address{Fl\'avio F. Cruz \\ Departamento de Matem\'atica,  Universidade Regional do Cariri, Campus Crajubar, 63041-141 -Juazeiro do Norte, Cear\'a (Brazil)}
\email{flavio.franca@urca.br}

\address{Barbara Nelli \\  Dipartimento di Ingegneria e Scienze dell’Informazione e Matematica,
Universit\`a dell’Aquila, via Vetoio Loc. Coppito, 67100 L’AQUILA (Italia)}
\email{barbara.nelli@univaq.it}

\keywords{Spherical cap conjecture; CMC surfaces; Convex surfaces}
\subjclass[2000]{53C42, 35J60}

\begin{abstract}
Liebmann's Theorem asserts  that a compact, connected, convex surface with constant mean curvature (CMC) in the Euclidean space must be a totally umbilical sphere. In this article  we  extend Liebmann's result to hypersurfaces with boundary. More precisely,  we prove that a locally convex, embedded, compact, connected CMC hypersurface bounded by a closed strictly convex  $(n-1)$-dimensional submanifold in a hyperplane $\Pi^n\subset \mathbb{R}^{n+1}$ lies in one of the two  halfspace  determined by $\Pi$ and inherits the symmetries of the boundary. Consequently, spherical caps are the only such hypersurfaces with non-zero constant mean curvature bounded by a $(n-1)-$sphere.
\end{abstract}

\maketitle

%%%%%%%%%%%%%%%%%%%%%%%%%%%%%%%%%%%%%%%%%%%%%%%%%%%%%%%%%%%%%%%%%%%%%%%%%
\section{Introduction}
\label{section1}

A  surface with constant mean curvature (CMC surface) is characterized by the fact that it is a critical point of the area functional with respect to local deformations that preserve the enclosed volume. Surfaces with this property serve as mathematical models for  soap bubbles. A soap film whose boundary lies on a round hoop is modeled by a compact CMC surface with circular boundary.  Experimental observation raises the question whether round spheres are the only closed CMC surfaces, and whether spherical caps and discs are the only compact CMC surfaces bounded by a circle -  the former for nonzero mean curvature, and the latter for zero mean curvature.
While the first question has been fully settled, the second remains wide open.

A notable result in this context, established by Liebmann  \cite{Lie00} in 1900, states that a closed and convex CMC surface must be a sphere. 
Later, Hopf \cite{Hopf3} proved that the sphere is the unique closed CMC surface with genus zero, immersed in the Euclidean space  and, soon after, Alexandrov  \cite{Ale56,Ale62, Hopf3} showed that the only closed CMC surface embedded  in the Euclidean space is the sphere. In the eighties, Barbosa and do Carmo  \cite{BC84} proved that an immersed  closed stable CMC surface must be a sphere.
Both Alexandrov and Barbosa - Do Carmo results hold for constant mean curvature hypersurfaces.

As we wrote above, when one considers  compact surfaces with non-empty boundary, it is still unknown whether the spherical caps are the only examples of compact embedded surfaces with non-zero constant mean curvature bounded by the circle. In  \cite{Kap91} Kapouleas showed that there exist examples of nonzero genus, compact immersed surfaces with constant mean curvature in $\mathbb{R}^3$ bounded by a circle. The so called \textit{Spherical cap conjecture}  was conjectured in \cite{BEMR91}.

\begin{conjecture}
	\label{conjecture}
	A compact non-zero CMC surface bounded by a circle is a spherical cap if either of the following conditions hold:
	
	i) The surface has genus zero and is immersed; 
	
	ii) The surface is embedded.
\end{conjecture}

%In \cite{BEMR91} the conjecture was proved, provided the surface is transversal to the plane containing the boundary.

In the spirit of Liebmann's result, we will prove  that   a locally convex, embedded, compact, connected CMC hypersurface bounded by $(n-1)$-dimensional sphere in $\mathbb{R}^{n+1}$ is necessarily a spherical cap. This provides a positive answer to the above conjecture, under the additional assumption of local convexity. To proceed, it’s worth noting that a hypersurface is said to be  {\em locally convex}  when all its principal curvatures are nonnegative.

Let us establish some notation.

Let $M^n$  be a compact embedded non-zero CMC hypersurface with boundary $\partial M=\Sigma$, a $(n-1)$-dimensional closed submanifold embedded in the hyperplane $\Pi^n$ of $\mathbb{R}^{n+1}$.  If $M$ is contained in one of the two halfspace determined by $\Pi$,  the Alexandrov reflection method  \cite{Ale56,Ale62} immediately proves $M$ has the reflectional symmetries of $\Sigma$ with respect to hyperplanes orthogonal to $\Pi$. Hence, if $\Sigma$ is a sphere, then $M$ is a hypersurface of revolution. Since the only compact CMC hypersurfaces of revolution are spherical caps (by Delaunay's classification of CMC surfaces of revolution in $\mathbb{R}^3$ and its extension to $\mathbb{R}^n$ obtained by W.-Y. Hsiang and W.-C. Yu in \cite{Hsiang}), Conjecture \ref{conjecture} holds for the subclass of surfaces that are embedded and contained in a halfspace. It is therefore of interest to obtain natural geometric conditions that force a compact embedded CMC  surface to be contained in a halfspace. In this direction, Koyso \cite{Koi86} proved that a sufficient condition for such an embedded non-zero CMC hypersurface $M$ with boundary $\partial M=\Sigma$ a $(n-1)$-dimensional closed submanifold embedded in a hyperplane $\Pi$ to be contained in a halfspace is to assume that $M$ does not intersect the outside of $\Sigma$ in $\Pi$. Later, Brito,  Meeks, S\'a Earp and Rosenberg showed in \cite{BEMR91} that if $M$ is a compact embedded CMC surface in $\mathbb{R}^{n+1}$ and $\partial M=\Sigma$  a convex submanifold embedded in a hyperplane $\Pi$, which is transverse to $M$ along the boundary $\partial M$, then $M$  lies on one side of $\Pi$. More recently, Al\'ias,  L\'opez and Palmer proved that the umbilical examples are the only CMC stable immersed discs bounded by a circle \cite{ALP99} in $\mathbb{R}^3$. Other interesting partial results on generalizations of  the Spherical Cap Conjecture to higher mean curvature functions and to Space Forms were obtained in \cite{AlDeMa06, Bar90, Bar91, BE91, Nelli}. We refer the reader to \cite{Lop1} for broad overview on the Spherical cap conjecture (see Chap. 1, 4, 5 and 7).

So far, no boundary analogue of Liebmann’s classical rigidity theorem for closed CMC surfaces has been established.
 The result below provides such an analogue.

\begin{theorem}
	\label{teorema}
	Let $\Sigma$ be a closed strictly convex  $(n-1)$-dimensional submanifold in a hyperplane $\Pi\subset \mathbb{R}^{n+1}$. Let $M$ be a connected compact embedded CMC hypersurface in $\mathbb{R}^{n+1}$ with boundary $\Sigma$. If $M$ is locally convex, then $M$ lies in one of the two halfspace defined by $\Pi$. In particular, if $\Sigma$ is a sphere, then $M$ is a $n$-ball or a spherical cap.
\end{theorem}

An outline of the proof is as follows. First, applying a constant rank theorem due Bian-Guan \cite{BG1} and a Simons' identity, we establish that a compact embedded locally convex CMC hypersurface $M$ bounded by a closed strictly convex  $(n-1)$-dimensional submanifold in a hyperplane $\Pi$ is  locally strictly convex in its interior and its asymptotic directions at the boundary are necessarily tangent to the boundary. This result allows us to prove that $M$  is contained in one of the two halfspace determined by $\Pi$ locally along the boundary. Finally, using the Alexandrov reflection method and the flux formula for CMC hypersurfaces we show that $M$ is contained  in a halfspace determined by $\Pi$, following the approach by  \cite{BE91, BEMR91}.

\subsection* {Acknowledgements}

Part of this work was completed while the first author was a guest at the Institut de Mathématiques de Jussieu–Paris Rive Gauche, whose hospitality is gratefully acknowledged. The first author also thanks CAPES for financial support through the CAPES/COFECUB Grant No. 88887.143177/2017-00. The second author is partially supported by INdAM-GNSAGA and 
 PRIN-2022AP8HZ9.

%%%%%%%%%%%%%%%%%%%%%%%%%%%%%%%%%%%%%%%%%%%%%%%%%%%%%%%%%%%%%%%%

\section{Boundary behavior}
\label{section2}

Let $M^n$ be a connected compact embedded CMC hypersurface in $\mathbb{R}^{n+1}$ with boundary $\partial M=\Sigma$, where $\Sigma$ is a closed strictly convex  $(n-1)$-dimensional submanifold in a hyperplane $\Pi^n\subset \mathbb{R}^{n+1}$. 

If the mean curvature of $M$ is zero $H=0$ then, by the Maximum Principle, $M=\Omega\subset \Pi$ is the bounded region in $\Pi$ limited by $\Sigma$. So, we assume that $H\neq 0$. It is worthwhile to remember that  $M$ is orientable and so, we may choose a globally defined unit vector field $N:M \rightarrow \mathbb{S}^n$ and assume $H>0$. 
First we prove the following:

\begin{proposition}
	\label{lema}
	Let $M$ be as in Theorem \ref{teorema}.  Then, $M$ is locally strictly convex on its interior and its asymptotic directions on the boundary are necessarily tangent to $\Sigma$. 
\end{proposition}
\begin{proof}
First we notice that $M$ has an interior elliptic point, that is, there is an interior point of $M$ where all the principal curvatures are positive. In fact, since $M$ is not part of a hyperplane, then we easily find a radius $R>0$ and a point $q\in \mathbb{R}^{n+1}$ such that the closed round ball $B_R(q)$ contains $M$ and there is a point $q_0\in  \textrm{int} (M)\cap\partial B_R(q)$ (englobe $M$ with spheres of large radius until such a sphere touches $M$ on one side at an interior point). Hence, the rank of the second fundamental form $h$  of $M$ is $n$ at $q_0$. Therefore, it follows by the constant rank theorem of Bian and Guan (Theorem 1.5 in \cite{BG1}) that the rank of $h$ is constant and equal to $n$ in the interior of $M$. Thus the first part of the proposition is proved. In order to prove the second part,
take a point $p\in \partial M$ such that $\det (h_{ij})(p)=0$, where $h=(h_{ij})$ are the components of $h$ in a local frame of $M$. We choose an orthornormal frame $\{e_1, \ldots e_n\}$ of $M$ around $p$ such that $e_1, \ldots, e_{n-1}\in T\Sigma$ and $e_n=\nu$ on $\Sigma$, where $\nu$ stands for the inner unit conormal vector of $M$ along $\Sigma$ and $h_{ij}(p)=0,$ if $i\neq j$ and $i,j<n$. In particular, $h_{ii}(p)>0$ for $i<n$ and,  as $\det (h_{ij})(p)=0$, one has
\begin{equation}
	\label{det}
h_{nn}(p)=\sum_{i<n}\frac{(h_{in})^2}{h_{ii}}(p).
\end{equation}
We claim that the multiplicity of first eigenvalue $\lambda_1(p)$ of $(h_{ij})(p)$ is one. 

In fact, if $h_{nn}(p)=0$ then, from \eqref{det}, one has $h_{in}(p)=0$, for any $i<n$. Then, for any $\xi=(\xi_1, \ldots ,\xi_n)\in \mathbb{R}^n\setminus\{0\}$, it follows that 
\begin{equation*}
	h_{ij}(p)\xi_i\xi_j =h_{ii}(p)\xi_i^2
\end{equation*}
is minimized (for $|\xi|=1$) when $\xi_i=0$ for $i<n$ and $\xi_n=1$. Therefore, the dimension of the eigenspace for $\lambda_1(p)=0$ is one. On the other hand,  if $h_{nn}(p)>0$ then
\begin{align*}
h_{ij}(p)\xi_i\xi_j & =\sum_{i<n}\left(h_{ii}(p)\xi_i^2+2h_{in}(p)\xi_i\xi_n+\dfrac{h_{in}^2(p)}{h_{ii}(p)}\xi_n^2\right) \\
	& =\sum_{i<n} \left(\sqrt{h_{ii}(p)}\xi_i+ \dfrac{h_{in}(p)}{\sqrt{h_{ii}(p)}}\xi_n \right)^2.
\end{align*}
Then, $ h_{ij}(p)\xi_i\xi_j$ is minimized when $\xi=(\xi_1,\ldots,\xi_n)$, with $|\xi|=1$, is given by
\begin{equation*}
	\xi_i=-\dfrac{h_{in}(p)}{h_{ii}(p)}\xi_n, \textrm{ for } i<n,\quad \textrm{and} \quad 	\xi_n=\left(1+\sum_{i<n}\left(\dfrac{h_{in}(p)}{h_{ii}(p)}\right)^2\right)^{-1} . 
\end{equation*}
Thus, the claim is proven. Now, let $\kappa_1\leq\cdots\leq  \kappa_n$ be the principal curvatures of $M$. Since $\kappa_1$ vanishes if and only if $\lambda_1$ vanishes, it follows that $\kappa_1(p)=0$ implies $\kappa_i(p)\geq \delta>0$ for $i\geq 2$. Therefore, we may choose  a small neighborhood $\mathcal{N}$ of $p$ in $M$  such that the smallest principal curvature $\kappa_1$ is smooth on $\mathcal{N}$. Now we need the following:
\begin{lemma}
	For a dense open set of $M$ we have
	\begin{equation}
		\label{EQ1}
\Delta \kappa_1 +|h|^2\kappa_1=nH\kappa_1^2-2\sum_{i}\sum_{j>1}\dfrac{(h_{1j,i})^2}{\kappa_j-\kappa_1}\leq nH\kappa_1^2,
	\end{equation}
where $\Delta$ denotes the Laplace-Beltrami operator of $M$.
\end{lemma}
\begin{proof}
	First we recall the Simons' identity for the Laplacian of the second fundamental form $h$ of a CMC hypersurface  (see Eq. (16) in \cite{KB})
	\begin{equation}
		\label{SimonEq}
		\Delta h+|h|^2h=nH h^2.
	\end{equation}
For a dense open set in $M$, we can choose a smooth orthonormal frame $e_1, \ldots,e_n$ of eigenvectors of $h=(h_{ij})$ corresponding to the ordered principal curvatures $\kappa_1\leq \cdots\leq  \kappa_n$ of $M$. See \cite{Sing} for the existence of such a frame. Thus, in term of this frame, one can rewrite \eqref{SimonEq} as
\begin{equation}
	\label{SimonEq2}
	\Delta h_{ij}+|h|^2h_{ij}=nHh_{ij}^2\delta_{ij}.
\end{equation}	
For $i=j=1$ one has 
	\begin{equation}
		\label{SimonEq3}
		\sum_{k}h_{11,kk}+|h|^2\kappa_1=nH\kappa_{1}^2.
	\end{equation}		
Since $\nabla_{e_k}e_1$ is orthogonal to $e_1$, for  $k\geq 2$, we have $\nabla_{e_k}e_1=\sum_{l>1}\langle \nabla_{e_k}e_1, e_l\rangle e_l$ and

 \begin{align}
	\begin{split}
	\label{conta-I}
	h_{11,k}  =& \left(\nabla_{e_k} h\right)(e_1, e_1)  = e_k(h(e_1,e_1))-h\left(\nabla_{e_{k}}e_1 , e_1\right)-h\left(e_1 , \nabla_{e_{k}} e_1\right)  \\ =& e_k(h_{11})-2\sum_{l>1}\langle \nabla_{e_k}e_1, e_l\rangle h\left(e_1 , e_l\right)\\ =& e_k(\kappa_{1}),
		\end{split}
\end{align}
since $h_{1l}=0$ for $l>1.$ 

Similarly,
\begin{align}
	\begin{split}
		\label{conta-II}
	h_{1\ell  ,k}    = & \left(\nabla_{e_k} h\right)(e_1, e_\ell)  = e_k(h_{1\ell})-h\left(\nabla_{e_{k}}e_1 , e_\ell\right)-h\left(e_1,\nabla_{e_{k}}e_\ell\right) \\ =& -\kappa_\ell\langle \nabla_{e_k}e_1,e_\ell\rangle-\kappa_1\langle \nabla_{e_k}e_\ell,e_1\rangle \\ =& (\kappa_1-\kappa_\ell)\langle \nabla_{e_k}e_1,e_\ell\rangle.
	\end{split}
\end{align}
Then
\begin{align}
	\begin{split}
		\label{conta-I|I}
h_{11, kk} = & \left(\nabla_{e_k} \nabla_{e_k}  h\right)(e_1, e_1) \\  = & e_k \left( h_{11,k} \right)-2\left(\nabla_{e_k} h\right)(\nabla_{e_k}e_1, e_1) - \left(\nabla_{\nabla_{e_{k}} e_k} h\right)(e_1, e_1) \\   = & e_k\left(e_k(\kappa_1)\right)-2\sum_{\ell>1}\langle \nabla_{e_{k}}e_1, e_\ell\rangle h_{1\ell, k}-\left(\nabla_{e_{k}}e_k\right)(\kappa_1),
	\end{split}
\end{align}
since, by \eqref{conta-I}, we have
\[
 \left(\nabla_{\nabla_{e_{k}} e_k} h\right)(e_1, e_1) = \left(\nabla_{e_{k}}e_k\right)(\kappa_1)
\]
Using \eqref{conta-II} we get
\[
\sum_{\ell>1}\langle \nabla_{e_{k}}e_1, e_\ell\rangle h_{1\ell, k}= - \sum_{\ell>1}\dfrac{(h_{1\ell,k})^2}{\kappa_\ell-\kappa_1}.
\]
Therefore
\begin{equation*}
	\sum_{k} h_{11,kk}=\Delta \kappa_1+2\sum_k\sum_{\ell>1}\dfrac{(h_{1\ell,k})^2}{\kappa_\ell-\kappa_1},
\end{equation*}	
as 
\begin{align*}
\Delta\kappa_1&=\sum_k \nabla^2 \kappa_1(e_k,e_k)=\sum_ig\left( \nabla_{e_k}\nabla\kappa_1, e_k\right)-g\left(\nabla_{e_k}e_k,\nabla\kappa_1\right) \\ &= \sum_k e_k\left(e_k(\kappa_1)\right)-\left(\nabla_{e_{k}}e_k\right)(\kappa_1),
\end{align*}
where $g$ denotes the metric of $\Sigma$. Applying this into \eqref{SimonEq3} we obtain \eqref{EQ1}.
\end{proof}	 
From \eqref{EQ1} we have
\begin{equation}
	\label{EQ2}
	\Delta \kappa_1+(|h|^2-nH\kappa_1)\kappa_1\leq 0
\end{equation}
in $\mathcal N$. However,
\begin{equation}
	\label{EQ3}
	|h|^2-nH\kappa_1 =\sum_{i>1}\kappa_i\left(\kappa_i-\kappa_1\right)>0,
\end{equation}
then, from \eqref{EQ2}, we conclude that $\kappa_{1}\geq 0$ satisfies
\begin{equation}
	\label{EQ4}
	\Delta \kappa_1\leq 0 \quad \textrm{in}\quad \mathcal{N}.
\end{equation}

Since $\kappa_1(p)=0$, the Hopf boundary point lemma implies $|\nabla \kappa_1(p)|\neq 0$. As $\kappa_1$ is smooth in $\mathcal{N}$, by the Implicit Function Theorem, the set \[
\mathcal Z=\{q\in\mathcal{N} : \kappa_1(q)=0\}
\]
is a smooth submanifold of $M$ passing through $p\in \Sigma$. If $\mathcal Z$ is transversal to $\Sigma$ then we have a contradiction since $\textrm{rank}(h)=n$ in the interior of $M$. Thus, $\mathcal Z$ is tangent to $\Sigma$ at $p$. Then, we conclude  that the asymptotic directions at points on the boundary of $M$ are necessarily tangent to $\Sigma$.
\end{proof}

Now we are in condition to show that, along $\Sigma$,  $M$ is locally contained in one of the two halfspace of $\mathbb{R}^{n+1}$ determined by $\Pi.$

\begin{proposition}
	\label{lema2}
	Let $M$ be as in Theorem \ref{teorema}.  Then, $M$, along $\Sigma$, is locally contained in one of the two halfspace of $\mathbb{R}^{n+1}$ determined by $\Pi.$ 
	\end{proposition}

\begin{proof}
 First we need to introduce some notations. The orientation of $M$ induces a natural orientation on its boundary as follows: we say that a basis $\{e_1, \ldots, e_{n-1}\}$ for $T_p\Sigma$ is positively oriented if $\{e_1, \ldots, e_{n-1}, \nu\}$ is a positively oriented basis for $T_pM$,  where $\nu$ denotes the inward pointing unit cornormal vector field along $\Sigma$. Let $\eta$  be the unitary vector field normal to $\Sigma$ in $\Pi$ which points outward with respect to the compact domain in $\Pi$ bounded by $\Sigma$. We denote by $\xi$ the unitary vector  normal to $\Pi$  in $\mathbb{R}^{n+1}$ which is compatible with $\eta$ and with the orientation of $\Sigma$, i. e., such that $\{\eta, e_1, \ldots, e_{n-1},\xi\}$ is a positively oriented basis of $\mathbb{R}^{n+1}$. We can assume that $\Pi$ is the horizontal plane that passes through the origin of $\mathbb{R}^{n+1}$.

Let $\{e_1, \ldots, e_{n-1}\}$ be (locally defined) a positively oriented frame field along $\Sigma$. Using this frame, we can write $\nu = N\times e_1\times\cdots\times e_{n-1} $ and $\eta= e_1\times\cdots\times e_{n-1} \times\xi$, since $\det( e_1, \dots, e_{n-1},\nu,  N)=\det( e_1, \dots, e_{n-1}, \eta, \xi)=1.$ From these expressions we compute
\begin{align*}
\eta =& e_1\times\cdots\times e_{n-1} \times \xi \\ =& e_1\times\cdots\times e_{n-1} \times (\langle\xi, N\rangle N+\langle \xi, \nu\rangle \nu)\\ =& -\langle\xi, N\rangle \nu+\langle \xi, \nu\rangle N.
\end{align*}
where   $\bar\nabla$ denotes  the usual connection of $\mathbb{R}^{n+1}$. 
Thus,
\begin{equation}
		\label{conta-angulos}
\langle \eta ,\nu\rangle=-\langle \xi, N\rangle \quad \textrm{and}\quad \langle \eta , N\rangle=\langle \xi, \nu\rangle.
\end{equation}
A direct computation gives that the second fundamental form of $M$ along of $\Sigma$ satisfies
\begin{equation}
	\label{conta-h-bordo}
h(e_i,e_j)=  h_{\Sigma} (e_i,e_j)\langle \xi, \nu\rangle,
\end{equation}
for any $1\leq i,j<n$, where $h_\Sigma$ denotes the second fundamental form of $\Sigma$ as a submanifold of $\Pi$. In fact,  as
\[
\bar \nabla_{e_i}e_j=\sum_k^{n-1}\langle \bar \nabla_{e_i}e_j, e_k\rangle e_k+\langle\bar \nabla_{e_i}e_j, \eta\rangle \eta+\langle\bar \nabla_{e_i}e_j, \xi\rangle \xi
\]
for any $1\leq i,j<n$, we have
\[
h(e_i, e_j)=\langle\bar \nabla_{e_i}e_j, \eta\rangle\langle\eta,N\rangle + \langle \bar \nabla_{e_i}e_j, \xi \rangle \langle\xi, N\rangle.
\]
It then follows from $\langle \bar\nabla_{e_i}e_j, \xi \rangle =0$ that
\[
h(e_i,e_j)=\langle \bar \nabla_{e_i}e_j, \eta\rangle  \langle N,\eta\rangle,
\]
then \eqref{conta-h-bordo} follows from \eqref{conta-angulos}.

Now we parameterize a neighborhood of $\Sigma$ in $M$ by setting  
\[
X(p,t)=\exp_{p}\left(t\nu(p)\right),
\]
$(p,t)\in \Sigma\times [0,\epsilon),$ for $\epsilon>0$ sufficiently small. Let $p\in \Sigma$ an arbitrary point of $\Sigma$. We claim that
 \[
 f(t):=\langle X(p,t), \xi \rangle >0
 \]
  for  $0<t<\epsilon$, if $\epsilon>0$ is small enough. Indeed, it follows from \eqref{conta-h-bordo} that 
  \[
  f'(0)=\langle  \nu, \xi\rangle= \frac{h(e_i, e_i)}{h_\Sigma(e_i, e_i)}(p).
  \] 
  for any $1\leq i<n$. Then, if $h(e_i,e_i)(p)>0$ for some $1\leq i<n$, we have that $f(0)=0$ and $f'(0)>0$, hence the claim is proved. On the other hand, if $h(e_i,e_i)(p)= 0$ for all $1\leq i<n$ we have $f(0)=f'(0)=0$ and follows from \eqref{conta-h-bordo} that $N(p)=\xi$. So,
  \[
  f''(0)= \langle X_{tt}, \xi \rangle (p)= \langle \bar\nabla_\nu \nu, N \rangle (p)=  h(\nu, \nu) >0
  \]
since $h(\nu, \nu) >0$,  by Proposition \ref{lema}. Thus, $f>0$ for  $0<t<\epsilon$, if $\epsilon>0$ is sufficiently small.
Hence, $M$ is locally contained in the halfspace $\Pi_+:=\{ p\in\mathbb{R}^{n+1} : \langle \xi, p\rangle >0\}$ along $\Sigma$.
\end{proof}

%%%%%%%%%%%%%%%%%%%%%%%%%%%%%%%%%%%%%%%%%%%%%%%%%%%%%%%

\section{Proof of Theorem \ref*{teorema}}

We start by proving that  a hypersurface $M$ satisfying the assumptions of Theorem \ref{teorema},  is contained entirely  in $\Pi_+$.
Assume by contradiction   that $M$, satisfying the assumptions of Theorem \ref*{teorema}, does not live entirely  in $\Pi_+$, that is, it  meets $\Pi$ other than at $\Sigma$. As $M$ is locally contained in $\Pi_+$ along $\Sigma$, we can assume $M$ meets $\Pi$ transversally by making a small vertical translation. Let $\Omega$ be the bounded region in $\Pi$ with boundary $\Sigma$.  When $M$ is a surface, it is proved in   \cite{BE91} and \cite{BEMR91} that transversality along $\Sigma$ yields that $M$ lives entirely in 
$\Pi_+$. The proof is analogous in all dimensions, but we include it here for the sake of completeness.

First we prove that $M\cap \Omega\neq \emptyset$ and $M\cap(\Pi\setminus \bar\Omega)=\emptyset$ leads to a contradiction. In fact, consider a hemisphere $S\subset\mathbb{R}^{n+1}$ below $\Pi$ with $\partial S \subset \Pi$ and denote by $B \subset \Pi$ the closed $n$-ball bounded by $\partial S$. Take $S$ of radius sufficiently large such that $\Omega \subset B$ and the part of $M$ below $\Pi$ is contained in the domain enclosed by $S\cup B$. If $A\subset \Pi$ denotes the annular region $B \setminus \Omega$, then $M\cup A\cup S$ defines a non-smooth, closed, embedded hypersurface in $\mathbb R^{n+1}$. By the Alexander duality, this hypersurface bounds a domain $W \subset \mathbb R^{n+1}$. We orient $M$ by the unit normal vector $N$ pointing to the interior of $W$.
Denote by $p$ and $q$ a highest and lowest points of $M$ with respect to $\Pi$, respectively. In particular, $p, q \in int(M)$ and  $\langle q, \xi\rangle < 0 < \langle p, \xi\rangle$, where $\xi$ is the unitary upward vector  normal to $\Pi$  in $\mathbb{R}^{n+1}$.  Since $N(p)$ points to $W$, $N(p)= (0,\ldots, 0,-1)$. Compare $M$ with the (horizontal) tangent space $T_p M$ at $p$, where $T_p M$ is oriented pointing downwards. Since $T_p M$ is a minimal surface and $M$ lives below $T_ p M$ around $p$, the tangent principle implies $H(p)\geq 0$. As $H\neq 0$, then $H >0$ on $M$. Now, we work with the point $q\in M$. Since $N(q)$ points to $W$, we have $N(q)= (0,\ldots, 0,-1)$ (Figure \ref{figura1}). Compare the horizontal plane $T_q M$ with $M$ at $q$. Let us take on $T_q M$ the orientation pointing downwards. In this situation, $T_q M$ lives above $M$ around $q$ and the tangent principle now yields $H(q)\leq  0$, a contradiction.
\begin{figure}
	\begin{center}
		\begin{tikzpicture}
			[line width=2pt] \draw (3,0) arc (0:-180:3);
			\draw (-4,0) -- (-2,0);
			\draw (2,0) -- (4,0);
			\node[below][font=\LARGE] at (4.2,0) {$\Pi$};
			\node[below][font=\LARGE] at (2.5,-2) {$S$};
			\draw [cyan] plot [smooth, tension=1] coordinates { (-2,0) (-1.5, .5) (-.8,-.6) (0,2) (1,-1.5) (2,0)};
			\node[above][font=\LARGE] at (1,1) {$M$};
			\draw[->] (-.05, 2) -- (-.05, 1.2);
			\draw[->] (1.2, -1.6) -- (1.2, -2.2);
			\node[above] at (-.05, 2) {$p$};
			\node[above] at (1.2, -1.6) {$q$};
		\end{tikzpicture}
	\end{center}
	\caption{$M\cap \Omega\neq \emptyset$ and $M\cap(\Pi\setminus \bar\Omega)=\emptyset$ }
	\label{figura1}
\end{figure}

We will show that $M\cap(P\setminus \bar\Omega)\neq \emptyset $ leads to a contradiction. Let $\Sigma_1,\ldots,\Sigma_m\subset \Pi$ be the closed $(n-1)$-dimensional submanifolds of $M$ in $\Omega$ (if there are any).
For each $i$, with $1\leq i\leq m$, let $\Sigma^+_i(\epsilon)$ be the submanifold of $M$, near $\Sigma_i$, obtained by intersecting $M$ with the horizontal hyperplane $\Pi(\epsilon)$, at height $\epsilon$. Similarly, let $\Sigma^-_i (\epsilon)$ be the submanifold $M\cap P(-\epsilon)$ that is near to $\Sigma_i$. We form an embedded hypersurface $N$ by removing from $M$ the annuli region bounded by the $\Sigma^+_i(\epsilon)\cup \Sigma^-_i(\epsilon)$ and attaching the horizontal  $n$-balls $B^+_i\cup B^-_i$ bounded
by the $\Sigma^+_i(\epsilon)\cup \Sigma^-_i(\epsilon)$. Also we attach $\Omega$ to $M$ along $\Sigma$. To ensure that $N$ is embedded, one uses different values of $\epsilon$ when several $\Sigma_i$ are concentric. 

Let $\bar M$ be the connected component of $N$ that contains $\Sigma$. $\bar M$ separates $\mathbb R^{n+1}$ into two components and let $W$ be the closure of the bounded component. The nonsmooth points of $\bar M$ are on $\Sigma$ and  $\Sigma^{\pm}_i(\epsilon)$, and the mean curvature vector of $M$ points into $W$. Notice that this is possible since $\bar M$ is obtained by attaching $n$-balls to a smooth connected submanifold of $M$ (where the mean curvature vector is never zero) and the mean curvature vector extends across these $n$-balls. We orient $M$ by the mean curvature vector. Denote the set $\Pi\setminus \Omega$ by $ext(\Omega)$. 

Now we observe that $\bar M\cap  ext(\Omega)$ has no components that
are null homotopic in $ext(\Omega)$. To see this, suppose $\tilde\Sigma$ were such a component. Let $\ell$ be an infinite line segment, starting at a point of $\Omega$ and intersecting  $\tilde\Sigma$ in at least two points. Consider a family $Q(t), t<\infty$, of parallel vertical hyperplanes coming from infinity and orthogonal to $\ell$. Suppose the family $Q(t)$ intersects $\bar M\cap ext(\Omega)$ for the first time at $t=t_0$.
Continuing the movement of $Q(t_0)$ by parallel translation towards $\Sigma$, would produce for some $t_1 < t_0$ a point of tangential contact of $\bar M$ with the reflection of $\bar M\cap \left(\cup_{t_{1}\leq t\leq t_{0}} Q(t) \cap \bar M\right)$ in $Q(t_1)$, before reaching $\Sigma$. By the Alexandrov reflection principle, this would yield a hyperplane of symmetry of $\bar M$ with $\Omega$ on one side of the hyperplane; a contradiction. We remark that the Alexandrov reflection principle applies here because $\bar M$ bounds the compact region $W$ and the first point of contact of the symmetry with $\bar M$ occurs before $\Omega$, hence at a smooth point of $M$. We used the fact that $\Sigma$ is convex here to ensure that a symmetry of $\Sigma$ touches $\Sigma$ before the hyperplane reaches $\Sigma$.

Therefore, we can assume $Q(t)$ touches $\bar M\cap ext(\Omega)$ for the first time along a submanifold $\tilde\Sigma$ that is homotopic to $\Sigma$ in $ext(\Omega)$. Let $A$ be the annular region in $ext(\Omega)$ bounded by $\tilde\Sigma\cup \Sigma$. Observe first that $int(A)$ contains no components $\tilde\Sigma_1$ of $\bar M\cap ext(\Omega)$ that are homotopic to $\Sigma$ in $A$. This follows by using the reflection principle with vertical hyperplanes $Q(t)$ as above: a symmetry of $\bar M$ would intersect $\bar M$ for a first time before arriving at $\Omega$.

Hence the mean curvature vector along $\tilde\Sigma\cup \Sigma$ points into $A$; in particular, along $\Sigma$, it points towards $ext(\Omega)$. But this contradicts the flux formula
\[
\int_\Sigma \langle Y, \nu\rangle =2H\int_\Omega \langle Y, \xi_\Omega\rangle,
\]
where $Y = \xi$, so $\langle Y,\nu\rangle > 0$ along $\Sigma$. Since the mean curvature vector points towards $ext(\Omega)$ along $\Sigma$, then $\xi_\Omega = -e_3$; so the right side is -$2H (\textrm{vol}\Omega)$; a contradiction. Thus, $M\cap(\Pi\setminus \bar\Omega)=\emptyset$ and $M\cap \Omega = \emptyset$, following that $M$ lives in $\Pi_+$. Now,  we can employ the Alexandrov reflection method with vertical planes to prove that $M$ inherits all the symmetries of its boundary. Consequently, if $\Sigma$ is an $(n-1)$-sphere, it follows that $M$ is rotationally symmetric and thus a spherical cap (see \cite{Hsiang}).
\qed

%%%%%%%%%%%%%%%%%%%%%%%%%%%%%%%%%%%%%%%%%%%%%%%%%%%%%%%%%%%%%%%%


\begin{thebibliography}{99} 


\bibitem{Ale56}
Alexandrov, A. D., Uniqueness theorems for surfaces in the large I. \emph{Vestnik Leningrad. Univ.,} \textbf{11} (1956), 5--17.

\bibitem{Ale62}
Alexandrov, A. D., A characteristic property of spheres. \emph{Ann. Mat. Pura Appl.}, \textbf{58} (1962), 303--315.

\bibitem{ALP99}
Al\'ias, L. J., L\'opez, R., Palmer, B., Stable constant mean curvature surfaces with circular boundary. \emph{Proc. Am. Math. Soc.} \textbf{127} (1999), 1195--1200.

\bibitem{AlDeMa06}
Al\'ias, L. J., De Lira, J. H., Malacarne J.M. , Constant higher-order mean curvature hypersurfaces in Riemannian spaces.\emph{J. Inst. Math. Jussieu} \textbf{5} (2006) 4, 527--562.


\bibitem{Bar90}
Barbosa, J. L., Hypersurfaces of constant mean curvature on $\mathbb{R}^{n+1}$ bounded by an Euclidean sphere. \emph{Geometry and Topology II.} World Sci. Publ. Teaneck. (1990), 1--9.

\bibitem{Bar91}
Barbosa, J. L., Constant mean curvature surfaces bounded by a planar curve. \emph{Mat. Contemp.} \textbf{1}  (1991), 3--15.

\bibitem{BC84}
Barbosa, J. L., do Carmo, M. Stability of hypersurfaces with constant mean curvature. \emph{Math. Z.} \textbf{173}  (1984), 339 --353.

\bibitem{BG1}
Bian, B., Guan, P., A microscopic convexity principle for nonlinear partial differential equations.  \emph{Invent. Math.} \textbf{177} (2009), 300 --335.

\bibitem{BE91}
Brito, F., Sa Earp, R., Geometric configurations of constant mean curvature surfaces with planar boundary.  \emph{An. Acad. Bras. Ci\^enc.} \textbf{63} (1991), 5 --19.

\bibitem{BEMR91}
Brito, F., Meeks, W.H. III, Sa Earp, R., Rosenberg, H., Structure theorems for constant mean curvature surfaces bounded by a planar curve.
\emph{Indiana Univ. Math. J.} \textbf{40} (1991), 333--343.

\bibitem{CNSIV} 
Caffarelli, L.,  Nirenberg L., Spruck, J., 
The Dirichlet problem for nonlinear second-order elliptic equations IV: Starshaped compact Weingarten
hypersurfaces, \emph{Current Topics in P.D.E.,} ed. by Y. Ohya et al. (1986),. 1--26, Kinokunize Co., Tokyo.

\bibitem{CL58}
Chern, S.-s., Lashof, R. K., On the total curvature of immersed manifolds II. 
\emph{Michigan Math. J.} \textbf{5} (1958), 5--12.

\bibitem{Hopf3}
Hopf, E., \emph{Diferential geometry in the large. Lectures Notes in Mathematics}, vol. 1000. Springer, Berlin 1983.

\bibitem{HRS92}
Hoffman, D., Rosenberg, H., Spruck, J., Boundary value problems for surfaces of constant Gauss curvature.
\emph{Commun. Pure Appl. Math.,} \textbf{45} (1992), 1051--1062.

\bibitem{Hsiang}
Hsiang, W.-Y., Yu, W.-C.: A generalization of a theorem of Delaunay.
\emph{J. Differ. Geom.,} \textbf{16} (1981), 161--177.


\bibitem{Jel}
Jellet, J. H., Sur la surface dont la courbature moyenne est constant.
\emph{J. Math. Pures Appl.,} \textbf{18} (1853), 163--167.

\bibitem{Kap91}
Kapouleas, N., Compact constant mean curvature surfaces in Euclidean three-space. 
\emph{J. Differ. Geom.} \textbf{33} (1991), 683--715.

\bibitem{KB}
Katsumi, N., Brian, S., A formula of Simons' type and hypersurfaces of constant mean curvature. 
\emph{J. Differ. Geom.} \textbf{3} (1969), 367--377.

\bibitem{Koi86}
Koiso, M., Symmetry of hypersurfaces of constant mean curvature with symmetric boundary.
\emph{Math. Z.} \textbf{191} (1986), 567--574.

\bibitem{Lie00}
Liebmann, H., Über die Verbiebung der geschlossenen Flächen positiver Krümmung.
\emph{Math. Ann.}  \textbf{53} (1900), 91--112.

\bibitem{Lop1}
L\'opez, R., \emph{Constant mean curvature surfaces with boundary.} Springer-Verlag Berlin Heidelberg, 2013.

\bibitem{Nelli}
Nelli, B., Rosenberg, H., Some remarks on embedded hypersurfaces in hyperbolic space of constant mean curvature and spherical boundary. \emph{Ann. of Glob. Anal. and Geom.}
\textbf{13} (1995), 23--30.

\bibitem{SERR}
Serrin, J., On surfaces of constant mean curvature which span a given space curve.
\emph{Math. Z.}. \textbf{112} (1969), 77--88.

\bibitem{Sing}
Singley, D., Smoothness theorems for the principal curvature and principal vectors of a hypersurface.
\emph{The Rocky Mountain J}. \textbf{5} (1975), 135--144.

\bibitem{Wen86}
Wente, H. C.,  Counterexample to a conjecture of H. Hopf.
\emph{Pac. J. Math.,} \textbf{121} (1986), 193--243.

\end{thebibliography}
\end{document}